\documentclass[11pt,leqno]{article}
\usepackage{graphicx, amsfonts, amsthm, amsxtra, amssymb, verbatim, makeidx}
\usepackage{subeqnarray, relsize}
\usepackage[mathscr]{euscript}
\usepackage{hyperref, tikz-cd}
\usepackage{aliascnt}
\usepackage{booktabs}
\usepackage{stmaryrd} 
\usepackage{soul} 
\usepackage[nameinlink,capitalize,noabbrev]{cleveref}
\hypersetup{
    colorlinks=true,       
    linkcolor=blue,          
    citecolor=blue,        
    filecolor=blue,      
    urlcolor=blue           
}

\newtheorem{theo}{Theorem}
\newtheorem{coro}{Corollary}

\newtheorem{prop}{Proposition}
\newtheorem{conj}{Conjecture}

\theoremstyle{remark}
\newtheorem{rem}{Remark}

\theoremstyle{definition}
\newtheorem{defi}{Definition}
\newtheorem{exam}{Example}

\makeindex
\makeglossary
\begin{document}

\def\U{{\cal U}}   
\def\res{{\rm res}}
\def\cf{r}   

\def\Nn{{\sf n}}
\def\Am{{\sf A}}

\def\Z{\mathbb{Z}}                   
\def\Q{\mathbb{Q}}                   
\def\C{\mathbb{C}}                   
\def\N{\mathbb{N}}                   
\def\Ff{\mathbb{F}}                  
\def\uhp{{\mathbb H}}                
\def\A{\mathbb{A}}                   
\def\dR{{\rm dR}}                    
\def\F{{\cal F}}                     
\def\Sp{{\rm Sp}}                    
\def\Gm{\mathbb{G}_m}                 
\def\Ga{\mathbb{G}_a}                 
\def\Tr{{\rm Tr}}                      
\def\tr{{{\mathsf t}{\mathsf r}}}                 
\def\spec{{\rm Spec}}            
\def\proj{{\rm Proj}}
\def\ker{{\rm ker}}              
\def\GL{{\rm GL}}                

\def\k{{\sf k}}                     
\def\ring{{\sf R}}                   
\def\sk{{\mathfrak k }}             
\def\sring{{\mathfrak R }}          

\def\X{{\sf X}}                      
\def\T{{\sf T}}                      
\def\V{{    V}}                   

\def\Ts{{\sf S}}
\def\cmv{{\sf M}}                    
\def\BG{{\sf G}}                       
\def\podu{{\sf pd}}                   
\def\ped{{\sf U}}                    
\def\per{{\sf  P}}                   
\def\gm{{\sf  A}}                    
\def\gma{{\sf  B}}                   
\def\ben{{\sf b}}                    

\def\Rav{{\mathfrak M }}                     
\def\Ram{{\cal C}}                         
\def\Rap{{i(\Lie(\BG))}}                    

\def\nov{{  n}}                    
\def\mov{{  m}}                    
\def\Yuk{{\sf Y}}                     
\def\Ra{{\sf R}}                      

\def\Da{{\sf D}}                      

\def\hn{{\sf h}}                      
\def\cpe{{\sf C}}                     
\def\g{{\sf g}}                       
\def\t{{   t}}                       
\def\v{{   v}}                       

\def\pedo{{\sf  \Pi}}                  

\def\Der{{\rm Der}}                   
\def\MMF{{\sf MF}}                    
\def\codim{{\rm codim}}                
\def\dim{{\rm    dim}}                
\def\Lie{{\rm Lie}}                   
\def\gg{{\mathfrak  g}}                

\def\u{{\sf u}}                       

\def\imh{{  \Psi}}                 
\def\imc{{  \Phi }}                  
\def\stab{{\rm Stab }}               
\def\Vec{{\Theta}}                 
\def\prim{{0}}                  
\def\Zero{{\rm Zero}}                  

\def\Fg{{\sf F}}     
\def\hol{{\rm hol}}  
\def\non{{\rm non}}  
\def\alg{{\rm alg}}  
\def\an{{\rm an}}   
\def\for{{\rm for}}  

\def\bcov{{\rm \O_\T}}       

\def\leaves{{\cal L}}        

\def\Hse{{\rm HS}}        
\def\Hpo{{\rm HP}}        
\def\Hfu{{\rm HF}}        
\def\Hsc{{\rm Hilb}}     

\def\TS{\mathlarger{{\bf T}}}                
\def\IS{\mathlarger{{\cal I}}}                

\def\vf{{\sf v}}                      
\def\wf{{\sf w}}                      

\def\red{{\rm red}}                           

\def\Ua{{   L}}                      
\def\plc{{ Z_\infty}}    

\def\gru{\mu} 
\def\pg{{ \sf S}}               
\def\group{{ G}}            

\def\GM{{\rm GM}}

\def\perr{{\sf q}}        
\def\perdo{{\cal K}}   
\def\sfl{{\mathrm F}} 
\def\sp{{\mathbb S}}  

\newcommand\diff[1]{\frac{d #1}{dz}} 
\def\End{{\rm End}}              

\def\sing{{\rm Sing}}            
\def\cha{{\rm char}}             
\def\Gal{{\rm Gal}}              
\def\jacob{{\rm jacob}}          
\def\tjurina{{\rm tjurina}}      
\newcommand\Pn[1]{\mathbb{P}^{#1}}   
\def\P{\mathbb{P}}                   
\def\Ff{\mathbb{F}}                  

\def\O{{\cal O}}                     
\def\as{\mathbb{U}}                  
\def\ring{{\mathsf R}}                         
\def\R{\mathbb{R}}                   

\newcommand\ep[1]{e^{\frac{2\pi i}{#1}}}
\newcommand\HH[2]{H^{#2}(#1)}        
\def\Mat{{\rm Mat}}              
\newcommand{\mat}[4]{
     \begin{bmatrix}
            #1 & #2 \\
            #3 & #4
       \end{bmatrix}
    }                                
\newcommand{\matt}[2]{
     \begin{bmatrix}                 
            #1   \\
            #2
       \end{bmatrix}
    }
\def\cl{{\rm cl}}                

\def\hc{{\mathsf H}}                 
\def\Hb{{\cal H}}                    
\def\pese{{\sf P}}                  

\def\PP{\tilde{\cal P}}              
\def\K{{\mathbb K}}                  

\def\M{{\cal M}}
\def\RR{{\cal R}}
\newcommand\Hi[1]{\mathbb{P}^{#1}_\infty}
\def\pt{\mathbb{C}[t]}               
\def\W{{\cal W}}                     
\def\gr{{\rm Gr}}                
\def\Im{{\rm Im}}                
\def\Re{{\rm Re}}                
\def\depth{{\rm depth}}
\newcommand\SL[2]{{\rm SL}(#1, #2)}    
\def\sl{{\rm SL}}                    
\newcommand\PSL[2]{{\rm PSL}(#1, #2)}  
\def\Resi{{\rm Resi}}              

\def\L{{\cal L}}                     
\def\Aut{{\rm Aut}}              
\def\any{R}                          
\newcommand\ovl[1]{\overline{#1}}    

\newcommand\mf[2]{{M}^{#1}_{#2}}     
\newcommand\mfn[2]{{\tilde M}^{#1}_{#2}}     

\newcommand\bn[2]{\binom{#1}{#2}}    
\def\ja{{\rm j}}                 
\def\Sc{\mathsf{S}}                  
\newcommand\es[1]{g_{#1}}            
\newcommand\WW{{\mathsf W}}          
\newcommand\Ss{{\cal O}}             
\def\rank{{\rm rank}}                
\def\Dif{{\cal D}}                   
\def\gcd{{\rm gcd}}                  
\def\zedi{{\rm ZD}}                  
\def\BM{{\mathsf H}}                 
\def\plf{{\sf pl}}                             
\def\sgn{{\rm sgn}}                      
\def\diag{{\rm diag}}                   
\def\hodge{{\rm Hodge}}
\def\HF{{\sf F}}                                
\def\WF{{\sf W}}                               
\def\HV{{\sf HV}}                                
\def\pol{{\rm pole}}                               
\def\bafi{{\sf r}}
\def\Id{{\rm Id}}                               
\def\gms{{\sf M}}                           
\def\Iso{{\rm Iso}}                           

\def\hl{{\rm L}}    
\def\imF{{\rm F}}
\def\imG{{\rm G}}

\def\HL{{\rm Ho}}     
\def\NLL{{\rm NL}}   

\def\RG{{\bf G}}          
\def\rg{{\bf g}}     
\def\rbullet{{\cdot}}
\def\Ld{{\cal L}}      
\def\Ro{{\rm R}}     
\def\ZS{{\rm ZeSc}}     
\def\ZI{{\rm ZeId}}     
 \def\integ{{\rm Int}}  

\def\tmap{{\sf t}}

\def\ivhs{{\rm IVHS}}    
\def\ivhsmaps{{{\Delta}_{}}}   
\def\sch{{\rm Sch}}   
\def\mk{{\mathfrak  m}}   
\def\pk{{\mathfrak  p}}   
\def\qk{{\mathfrak  q}}   

\newcommand\licy[1]{{\mathbb P}^{#1}} 

\def\SS{\mathscr{S}}    

\begin{center}
{\LARGE\bf Hasse-Witt invariants of Calabi-Yau varieties \footnote{\today}
}
\\
\vspace{.1in} {\large {\sc Jin Cao\footnote{School of Mathematics and Physics, University of Science and Technology Beijing, Beijing, China}, Mohamed Elmi\footnote{Yau Mathematical Sciences Center, Tsinghua University, Beijing, China} and Hossein Movasati}}
\footnote{
Instituto de Matem\'atica Pura e Aplicada, IMPA, Estrada Dona Castorina, 110, 22460-320, Rio de Janeiro, RJ, Brazil,
{\tt \href{http://w3.impa.br/~hossein/}{www.impa.br/$\sim$hossein}, hossein@impa.br.}}
\end{center}


\begin{abstract}
We define the Hasse-Witt invariant of Calabi-Yau varieties in two different ways. The first method is through Cartier operator and the second method is through the theory of Calabi-Yau modular forms developed by the third author. We conjecture that these two definitions are equivalent and provide many examples of Calabi-Yau varieties in support of this conjecture. 
\end{abstract}

\section{Introduction}
Let $X$ be an elliptic curve over a perfect field $\sk$ of characteristic $p\not=2,3$ and let $\alpha$ be a regular differential $1$-form on $X$. The Cartier operator $C$ acts on the sheaf of differential $1$-forms on $X$ and the Hasse-Witt invariant is defined by the equality
$$
C(\alpha)={\rm HW}(X,\alpha)^{\frac{1}{p}}\alpha.
$$
The algebra of (elliptic) modular forms for $\SL 2\Z$ gives us another method for computing the Hasse-Witt invariant. It is well-known that this algebra is generated by the Eisenstein series $E_4$ and $E_6$. In particular, this implies that 
\begin{equation}
\label{19122025lastday}
E_{p-1}=A_p\left(\frac{1}{12}E_4,\frac{-1}{216}E_6\right)
\end{equation}
for some polynomial $A_p\in\Q[x,y]$ of weighted degree $p-1$ where $\deg(x)=4,\ \deg(y)=6$ . It turns out that $p$ does not appear in the denominators of the coefficients of $A_p$, we can reduce $A_p$ modulo $p$. If we write the pair $(X,\alpha)$ in the Weierstrass coordinates 
 \begin{equation}
 \label{25112025alirezaakbari}
 \X_{t_2,t_3}: y^2=4x^3-t_2x-t_3,\ \alpha=\frac{dx}{y}
 \end{equation} 
 then the following equality in characteristic $p$
 $$
{\rm HW(X,\alpha)}= A_p(t_2,t_3)
 $$
 is a classical result, mainly attributed to P. Deligne, see \cite[page 90]{ka73}. 

The first part of the above story was developed by N. Katz for arbitrary varieties. In recent years, there has been some interest in similar notions in the case of Calabi-Yau varieties e.g. \cite{Ogus2001}. 

The third author has developed the theory of modular forms for Calabi-Yau varieties (CY modular forms) in many articles with the title "Gauss-Manin connection in disguise", see \cite{ho2020, GMCD-MQCY3} and the reference therein. Much of this work has been motivated by the study of generating series of Gromov-Witten invariants, even though there is no well-known enumerative algebraic geometry for the CY modular forms of the present article. 

The theory of CY modular forms generalizes the theory of quasi-modular forms $\Q[E_2,E_4,E_6]$. In this article, we deal only with part of the algebra of CY modular forms which generalizes $\Q[E_4,E_6]$.  It turns out that $\spec\left(\Z\left[\frac{1}{6(27t_3^2-t_2^3)}, t_2,t_3\right]\right)$ is the moduli scheme of pairs $(X,\alpha)$  and it gives us a geometric incarnation of the classical modular forms. 

We consider the moduli $\Ts$ of  pairs $(X,\alpha)$, where $X\subset\P^N_\C$ is a projective/polarized Calabi-Yau $n$-fold with a fixed Hilbert polynomial $P$, which varies in a component of the  Hilbert scheme ${\rm Hilb}_P(\P^N)$,  and $\alpha$ is a 
holomorphic differential $n$-form on $X$. 
The multiplicative group $\Gm:=(\C^*,\cdot)$ acts on $\Ts$ by:
 $$
(X,\alpha)\bullet a=(X,a^{-1}\alpha),\ a\in \Gm,\ (X,\alpha)\in \Ts.
$$
Let $L$ be the line bundle on the classical moduli space of Calabi-Yau $n$-folds given by holomorphic $n$-forms $\alpha$. The moduli space $\Ts$ is not the total bundle of $L$. It is such a total bundle quotiented by action of automorphisms $X$ on $\alpha$. We expect that 
\begin{conj}
\label{17102025bimsa-1}
The moduli space $\Ts$ has a natural affine scheme structure over $\Z[\frac{1}{N}]$ for some $N\in\N$. We have global regular functions $t_i$ on $\Ts$  and $k_i\in\Z$ such that
\begin{equation}
\label{02122025bri}
t_i(X,a^{-1}\alpha)=a^{k_i}t_i(X,\alpha), 
\end{equation}
the homogeneous polynomial $\Delta\in \Z[\frac{1}{N},t]:= \Z[\frac{1}{N},t_1,t_2,\cdots, t_s],\ \ \deg(t_i)=k_i$, called the discriminant,  such that 
$$
\Ts=\spec\left (\Z\left[ \frac{1}{N\Delta},t\right]/I\right ).
$$
Here, $I$ is a homogeneous ideal in $\Z[\frac{1}{N\Delta},t])$. Moreover, the moduli space is fine, that is, we have a universal family $\X\to\Ts$ which comes together with $\alpha\in H^0(\X,\Omega^n_{\X/\Ts})$. 
\end{conj}
We denote a fiber of $\X\to \Ts$ over $t\in\Ts$ by $\X_t$ which comes together with a unique differntial $n$-form $\alpha_t$. 
A global regular section of $\Ts$ is called a (meromorphic) CY modular form. 
E. Viehweg has constructed the coarse  moduli space of polarized smooth Calabi-Yau varieties defined over algebraically closed fields of characteristic zero as a quasi-projective variety, see \cite{viehweg95}. It does not seem to be difficult to manipulate his construction and prove \cref{17102025bimsa-1}. 
Note that once this conjecture is established, the classical moduli of Calabi-Yau varieties without the differential $n$-form $\alpha$ is the variety given by $I=0$ in the weighted projective space $\P^{k_1,k_2,\ldots, k_s}_\Q\backslash\{\Delta=0\}$. 

The above definition of CY modular forms does not indicate at all what $q$-expansion means. 
This is as follows. Let $X_z,\ z\in \C^m$ be a family of Calabi-Yau $n$-fold with $m:=h^{n-1,1}$ and assume that $z=0$ is a MUM point, see \cite{Morrison1993} for the definition.  Let $\alpha_z$ be a holomorphic $n$-form on $X_z$. We assume that all these are defined over $\Q$. It is well-known that we have cycles $\delta_{i,z}\in H_n(X_z,\Z),\ \ \ i=1,2,\ldots,m+1$ such that 
$$
\int_{\delta_{0,z}}\alpha_z,\ \ \int_{\delta_{i,z}}\alpha_z-\frac{\ln(z_i)}{2\pi\sqrt{-1}}\int_{\delta_{0,z}}\alpha_z,\ \ i=1,2,\ldots,m
$$ 
are holomorphic and the mirror map 
$$
(\C^m,0)\to(\C^m,0),\ \ \ (z_1,z_2,\ldots,z_m)\mapsto (q_1,q_2,\ldots, q_m),\ \ 
q_i:=\exp\left(\frac{\int_{\delta_{i,z}}\alpha_z}{\int_{\delta_{0,z}}\alpha_z}\right)
$$
is a biholomorphism, see \cite{Morrison1993}, and hence, we can invert it $z=z(q)$. We now take a smooth model of $X_z$ over $\Z[\frac{1}{N}]$ for some $N\in\N$, and  take $z$ from a perfect field $\sk$ of characteristic $p\nmid N$ and take $(X_z, \alpha_z)$ defined over $\sk$. We have the Cartier operator $C$ acting on closed differential forms, and $C(\alpha_z)$ is also a global holomorphic $n$-form on $X_z$. Therefore, 
$$
C(X_z, \alpha_z)={\rm HW}(X_z,\alpha_z)^\frac{1}{p}\alpha_z,\ \ \ a(z):={\rm HW}(X_z,\alpha_z)
$$
where ${\rm HW}(X_z,\alpha_z)\in \sk$ is called the Hasse-Witt invariant. It turns out that in all our examples  the Hasse-Witt invariant is a polynomial  in $z$ with coefficients in $\Ff_p$. From another side we have the period integral 
$$
{\rm HP}(X_z,\alpha_z):=
\int_{\delta_{0,z}}\alpha_z\Big / \int_{\delta_{0,0}}\alpha_0
$$
which is the unique holomorphic solution to the Picard-Fuchs system of $(X_z,\alpha_z)$ with the constant term $1$, and moreover, it is $p$-integral, and hence we can take its reduction modulo $p$. It turns out that both quantities ${\rm HP}(X_z,\alpha_z)$ and ${\rm HW}(X_z,\alpha_z)$ can be evaluated at $z=0$.

\begin{conj}\rm
\label{17112025ymsc}
For all primes $p\not| N$,
the Hasse-Witt invariant ${\rm HW}(X_z,\alpha_z)$ up to sign  is the truncation of the holomorphic period at degree $p-1$  and  the quantity
$$
A_p(z):={\rm HP}(X_z,\alpha_z)^{p-1}
{\rm HW}(X_z,\alpha_z)
$$
after inserting the mirror map $A_p(z(q))$ has $p$-integral coefficients and modulo $p$ it is a  constant independent of $z$, more precisely, $A_p(z(q))\equiv_p 1$ (up to sign).
\end{conj}
For the first part of Conjecture \ref{17112025ymsc}, it has been shown in the case of Calabi-Yau varieties realized as some hypersurfaces in a smooth toric variety in \cite{HLYY23} and some three-dimensional reflexive polytopes (see \cite[Main Theorem]{SalernoWhitcher22}). However the $p$-integrality of the mirror map itself, which implies the $p$-integrality of $A_p(z(q))$, is still an unsolved conjecture which has been settled mainly in the case of hypergeometric Calabi-Yau varieties using results of B. Dwork, see \cite[Appendix C]{GMCD-MQCY3} and references therein. This is going to be discussed in the third author's paper \cite{Ibiporanga} in which a larger moduli space (ibiporanga) as in \cite{ho2020, GMCD-MQCY3} is used.    

The first part of the above conjecture might not be so difficult as it only requires the study of the degeneration of Hasse-Witt invariant when $X_z$ becomes singular, see \cref{27112025bimsa}. However, note that in \cref{17112025ymsc} we cannot replace the holomorphic period with its truncated one, and hence consider $A_p(z)=a(z)^{p}$,  since the $q$-expansion has to do with the full period.

Let us now rewrite \cref{17112025ymsc} in the context of \cref{17102025bimsa-1}. Since $\Ts$ is the moduli space of the pairs $(X,\alpha)$, for each $z$ in a small neighborhood of $0\in\C^n$ with $z_i\not=0$, we have a unique $t(z):=(t_1(z),t_2(z),\ldots,t_s(z))\in\Ts$, where $t_i(z)$'s are rational functions in $z$,  such that 
$$
(X_z,\alpha_z)\cong (\X_{t(z)},\alpha_{t(z)}).
$$
We define a special locus
$$
\tilde\Ts:=\left\{ t\in \Ts \ \ \ \ \Big| \ \ \ \ 
\int_{\delta}\alpha_t
=c,\ \ \hbox{for some cycle} \ \  \delta\in H_n(\X_t,\Z) \right \}.
$$
where $c:=\int_{\delta_{0,0}}\alpha_0$ as above. It plays the similar role as the $\tau$-locus defined in \cite[Section 4.2]{GMCD-MQCY3}.
Note that $ \int_{\delta}\alpha_t$ is a holomorphic multi-valued function in $\Ts$. 
By the functional equation of $t_i$'s in \eqref{02122025bri} it follows that such a loci in $t_i$-coordinates is given by 
$$
\left ({\rm HP}(X_z,\alpha_z)^{k_i}t_i(z),\ \ i=1,2,\ldots,s\right )\in\Ts.
$$
After inserting the mirror map we get $q$-expansions of $t_i$'s that we denote by $t_i(q)$. This is the holomorphic incarnation of the algebraic CY modular forms $t_i$ as a regular function on $\Ts$. \cref{17112025ymsc} turns into the following: 
\begin{conj} \label{02112025huairou}
For any prime $p\nmid N$ we have a homogeneous polynomial $A\in \Ff_p[t]$ of degree $p-1$ with $\deg(t_i)=k_i$ such that $C(\alpha)=A_p(t)^{\frac{1}{p}}\alpha$, where $C$ is the Cartier operator. Moreover, if we write the $q$-expansion of $t_i$'s then 
$$
A_p(t_1(q),t_2(q),\cdots, t_s(q))\equiv_p1. 
$$
\end{conj}
We recall that one way to prove \cref{02112025huairou} for elliptic curves is to observe that $A_p(t_2,t_3)$ is a Hecke eigenform of weight $p-1$ with the constant term $1$. This determines $A_p$ uniquely. Unfortunately, all the efforts of the third author to generalize Hecke operators for Calabi-Yau modular forms has failed, see \cite[Section 11.2]{GMCD-MQCY3}. 
We are able to automate a verification of our conjetures and find that

\begin{theo}
\label{28112025bri}

\cref{17102025bimsa-1}, \cref{17112025ymsc} and \cref{02112025huairou} are true for the first $200$ primes and to $O(q^{200})$ for four families of hypergeometric Calabi-Yau threefolds in  
\cite{Morrison:1991cd}. We have $\Ts_\Q=\spec(\Q[t_1, t_k, \frac{1}{k(t_1^k-t_k)}])$. The universal family of hypersurfaces in the weighted projective space $\mathbb{P}^{k_0,k_1,k_2,k_3,k_4}$ is given in table below:

\begin{equation*}
\begin{array}{|c|c|c|}
\hline
\cite{alenstzu} \# & (k_0,k_1,k_2,k_3,k_4) &  f\\[2pt]
   \hline
1 & (1,1,1,1,1) & t_5x_0^5 + x_1^5 + x_2^5+x_3^5 + x_4^5 -t_1x_0x_1x_2x_3x_4\\[2pt]
2 & (5,2,1,1,1) & x_0^2 + x_1^5 + t_{10}x_2^{10} + x_3^{10} + x_4^{10} - t_1 x_0x_1x_2x_3x_4\\[2pt]
7 & (4,1,1,1,1) & x_0^2 + t_8x_1^8 + x_2^8 + x_3^8 + x_4^8 - t_1x_0x_1x_2x_3x_4\\[2pt]
8 &   (1,1,1,2,1) & t_6x_0^6 + x_1^6 + x_2^6 + x_3^3 + x_4^6 - t_1x_0x_1x_2x_3x_4\\[2pt]
\hline
\end{array}
\end{equation*}
The differential $3$-form is given by 
$\alpha=\frac{dx_1\wedge dx_2\wedge dx_3\wedge dx_4}{df}$, 
where $f$ is the equation of the Calabi-Yau threefold  with $x_0=1$. 

\end{theo}



We return to the case of elliptic curves and explain \cref{17112025ymsc} for the Legendre family $E_z:y^2=x(x-1)(x-z)$ together with the holomorphic 1-form $\alpha_z = \frac{dx}{y}$. The $1$-form $\alpha_z$ can be integrated over generators $\delta_{0,z}$ and $\delta_{1,z}$ of $H_1(E_z,\mathbb{Z})$. This leads to the  $A$ and $B$ periods
\begin{align*}
\int_{\delta_{0,z}}\alpha_z &= 2\pi F\left(\frac{1}{2} \frac{1}{2},1; z\right)=
2\pi\sum_{k=0}^\infty \frac{(\frac{1}{2})_k (\frac{1}{2})_k}{(k!)^2} z^k
\\
\int_{\delta_{1,z}} \alpha_z &= \frac{2}{\sqrt{-1}}\left(F\left(\frac{1}{2} \frac{1}{2},1; z\right)\ln\left(\frac{z}{16}\right)+G(z)\right),\ \ \mathbb{ }G(z)=4\sum_{k=1}^\infty 
 \frac{(\frac{1}{2})_k (\frac{1}{2})_k}{(k!)^2}
(H_{2k}-H_k)z^k.
\end{align*}
where $H_k$ is the $k^\text{th}$ Harmonic number. See, for example, \cite{Zagier2018ArithmeticTopologyDE}. The periods are annihilated by the Picard-Fuchs operator
\[
\theta^2-z\left(\theta+\frac{1}{2}\right)^2=0,\ \ \ \ \theta:=z\partial_z
~\]
and, by construction, they have monodromy in $SL(2,\mathbb{Z})$. 

Let $p$ be a prime bigger than $3$. The Hasse-Witt invariant of the Legendre family is simply the truncation of the holomorphic period: 
\[
a(z) = (-1)^{\frac{p-1}{2}}
\sum_{k=0}^\frac{p-1}{2} \frac{(\frac{1}{2})_k (\frac{1}{2})_k}{(k!)^2} z^k,
\]
which, modulo $p$, is the coefficient of $x^{p-1}$ in the expression $(x(x-1)(x-z))^{\frac{p-1}{2}}$. 

We define $\tau(z)$ as the ratio of $B$ and $A$ periods and compute the mirror map $q(z)=\frac{z}{16}\exp\left(\dfrac{G(z)}{F(z)} \right)
$.\footnote{Note that $q=e^{\pi i \tau}$.} As expected, we recover the $q$-expansion of the modular lambda function when we invert this power series i.e.
\begin{equation*}
    z(q)=\frac{\theta^4_2(q)}{\theta^4_3(q)}=16 q-128 q^2+704 q^3-3072 q^4+11488 q^5-38400
   q^6+O\left(q^7\right),
\end{equation*}
We verify that, for the first $200$ primes,
$$
A_p(z(q)) = F\left(\frac{1}{2} \frac{1}{2},1; z(q)\right)^{p-1}a(z(q))=(-1)^{\frac{p-1}{2}} + \mathcal{O}(q^{200})  \mod p.
$$
If we consider the moduli space $\Ts$ of $(E,\alpha)$ as before, the previous procedure gives us 
$$
\left(z(q)^2-z(q)+1\right) F\left(\frac{1}{2} \frac{1}{2},1; z(q)\right)^4=E_4(q^2),\ 
$$
$$
\left(z(q)^3-\frac{3}{2}z(q)^2-\frac{3}{2}z(q)+1\right)F\left(\frac{1}{2} \frac{1}{2},1; z(q)\right)^6=E_6(q^2).
$$
\begin{rem}
It is known that:
$$
F\left(\frac{1}{2} \frac{1}{2},1; z(q)\right) = \theta_3^2(q), \ \ \ \theta_3(q) = \sum^{\infty}_{k=-\infty} q^{k^2}.
$$
For example, see \cite[(4.32)]{Wenzhe2021} or \cite[Theorem 8.3]{Chan2020}.
\end{rem}
In a similar way, if we consider the moduli of $(E,\alpha,P,Q)$, where $(E,\alpha)$ as before, and $P$ and $Q$ are 2-torsions with Weil pairing $+1$ then the moduli space is $\spec(\Z[\frac{1}{6t_2s_2(t_2-s_2)},t_2,s_2])$ with the universal family $y^2=x(x-t_2)(x-s_2)$ and we get 
$$
\frac{1}{4}F\left(\frac{1}{2} \frac{1}{2},1; z(q)\right)^2:=q\frac{\partial}{\partial q}\ln(\frac{\theta_2(0|q)}{\theta_4(0|q)})
$$
and
$$
\frac{z(q)}{4}F\left(\frac{1}{2} \frac{1}{2},1; z(q)\right)^2=q\frac{\partial}{\partial q}\ln(\frac{\theta_3(0|q)}{\theta_4(0|q)}),
$$
which generate the algebra of  modular forms  for $\Gamma(2)$, see \cite[Page 334]{ho14}. 

{\bf Acknowledgment:} This article was written during the third author's visit to BIMSA and YMSC at Tsinghua University. We thank S.T.Yau for his efforts in creating these insitutions, where we could meet and discuss mathematics. The first author was supported by the University of Science and Technology Beijing Foundation, China (Grant No. 00007886) and the Fundamental Research Funds for the Central Universities (Grant No. 06320202). The second author would like to acknowledge support from the Shuimu Tsinghua Scholar Program and the China Postdoctoral Science Foundation.

\section{Cartier operator and Hasse-Witt invariant}
\label{27112025bimsa}
In this section we review Cartier operator and Hasse-Witt invariant from \cite{BK05}.
Let $X = \mathrm{Spec}(A)$ be an affine scheme over a perfect field $\sk$ of characteristic $p >0$ and $\Omega_A^1$ be the module of K\"{a}hler differentials of $A$ over $k$, which is equipped with the $\sk$-derivation
\[
d: A \to \Omega^1_A, a \to da.
\]
Then the de Rham complex $(\Omega_A^{\bullet}, d)$ of $A$ consists of the exterior algebra $\Omega_A^{\bullet}$ of $\Omega_A^1$ over $A$ together with the extended maps $d: \Omega^{\bullet}_A \to \Omega^{\bullet}_A$. 
We let
\[
Z^i_A = \{\alpha \in \Omega_A^i \mid d\alpha = 0\},  \ \ \ B^i_A = d \Omega_A^{i-1}, H_A^i = Z^i_A/B^i_A.
\]
Then one may define an $A$-algebra homomorphism:
\begin{eqnarray*} \nonumber
    \gamma: \Omega^{\bullet}_A &\to& H^{\bullet}_A \\
    a_1da_2 \wedge \cdots\wedge da_n &\to& a_1^p a_2^{p-1}da_2 \wedge \cdots \wedge a_n^{p-1} da_n (\mathrm{mod}\ B_A^{\bullet}), 
\end{eqnarray*}
where the $A$-module structure on $H^{\bullet}_A$ comes from the Frobenius action on $A$, see \cite[Lemma 1.3.3]{BK05}. Now we consider $X$ is a scheme over $\sk$ and let $F$ be the relative Frobenius $X \to X^{(p)}$. Then the sheaf-theoretic version of the above construction provide us a unique homomorphism of sheaves of graded-commutative $\mathcal{O}_{X^{(p)}}$-algebras:
\[
\gamma: \Omega^{\bullet}_{X^{(p)}} \to \bigoplus^{\infty}_{i=0} \mathcal{H}^i F_{*} \Omega_X^{\bullet}
\]
which satisfies:
\begin{itemize}
    \item 
    $\gamma(f) = f^p$ for $f \in \mathcal{O}_{X^{(p)}}$;
    \item 
    $\gamma(df) = f^{p-1}df(\mathrm{mod} \ d\mathcal{O}_{X^{(p)}})$ for $f \in \mathcal{O}_{X^{(p)}}$;
    \item 
    For $\alpha, \beta \in \Omega^{\bullet}_{X^{(p)}}$, we have: $\gamma(\alpha + \beta) = \gamma(\alpha) + \gamma(\beta)$ and
    $\gamma(\alpha \wedge \beta) = \gamma(\alpha) \wedge \gamma(\beta)$.
\end{itemize}
We remark that $\gamma$ is an isomorphism when $X$ is a nonsingular variety. The proof is contained in \cite[Theorem 2.1.1]{Katz1972} or \cite[Theorem 1.3.4]{BK05}.
\begin{defi}
Let $X$ is a nonsingular variety of dimension $n$. The inverse of the isomorphism $\gamma$ is called the Cartier operator:
\[
C = \sum^n_{i=0} C_i: \bigoplus^{n}_{i=0} \mathcal{H}^i F_{*} \Omega_X^{\bullet} \to \Omega^{\bullet}_{X^{(p)}}.
\]
\end{defi}
\begin{rem}
For $X$ a nonsingular variety of dimension $n$, we define the Cartier operator for $n$ forms as the composition of the maps $C: F_* \Omega_X^n \to \mathcal{H}^n F_{*} \Omega_X^{\bullet} \to \Omega^{n}_{X^{(p)}}$ by abuse of notations.
\end{rem}

We let $X$ be a smooth family of nonsigular projective Calabi-Yau varieties of dimension $n$ over an integral affine variety $S = \mathrm{Spec}(R)$, where $R$ is a finite generated algebra over a perfect field $\sk$ of characteristic $p > 0$. Then $f_* \Omega_{X/S}$ is a locally free sheaf of rank $1$ over $S$, where $f: X \to S$. Now we fix a basis $\alpha$ of $f_* \Omega_{X/S}$, whose dual basis under the Serre duality is denoted by $\eta \in R^nf_* \mathcal{O}_X$. Note that the absolute Frobenius action $F_{\mathrm{abs}}: \mathcal{O}_X \to \mathcal{O}_X$ induces an endomorphism $F_{\mathrm{abs}}^*$ of $R^nf_* \mathcal{O}_X$. Then:
\begin{defi}
We define the Hasse-Witt invariant of $X$ with respect to $\alpha$ as the element $A_p(X, \alpha) \in R$ such that:
\[
F_{\mathrm{abs}}^*(\eta) = A_p(X, \alpha) \eta.
\]
\end{defi}
\begin{rem}
If we assume further that each fiber of $X$ is a hypersurface, then the Hasse-Witt invariant can be also defined as 
\[
C(F_* \omega) =  A_p(X, \alpha)^{\frac{1}{p}} \omega^{(p)}.
\]
The proof of this equality is based on the algorithm of the computation of Hasse-Witt invariants for hypersurfaces of Katz \cite[Algorithm 2.3.7.14]{Katz1972} and \cite[Corollary 1]{Miller72}, since both of these computational results lead to the same invariant.
\end{rem}

One remarkable property is that the Hasse-Witt invariant satisfies the Picard-Fuchs equation due to the Katz-Igusa-Manin theorem \cite[Proposition 2.3.6.3]{Katz1972}. Applying this result to the family of Calabi-Yau $n$-folds, we have:
\begin{theo}
\label{01122025zutopia}
Let $\mathcal{D}$ be $\sk$-linear differential operator on $S$, which is contained in the algebra generated by $\mathrm{Der}(S/\sk)$ and acts on the de Rham cohomology sheaves $f_* \Omega_{X/S}$ via the Gauss-Manin connection $\nabla$. Suppose that
\[
\nabla_{\mathcal{D}}(\alpha) = 0,
\]
where we view $\alpha$ as a section of $f_* \Omega_{X/S}$ as above.
Then $\mathcal{D}(A_p(X, \alpha)) = 0$.
\end{theo}
Note if ${\cal D}:= \sum p_i\frac{\partial^i}{\partial t^i}$ then by definition we have $\nabla_{\cal  D}=\sum p_i\nabla_{\frac{\partial}{\partial t}}^i$. From the above theorem it might be possible to prove the following: 
 The Hasse-Witt invariant of $X_z$ is up to multiplication by a constant  $c_2\in\Ff_p$ 
 the truncation at order $p-1$ of the Taylor sereis of the holomorphic period $\psi_0(z):=(2\pi i)^{-n}\int_{\delta_{0,z}}\alpha_{z}$ at the MUM point. The idea of the proof is as follows: 
 Let us consider the differential sub module $L$ of $\Z[\frac{1}{N}][z_1,z_2,\ldots,\theta_1,\theta_2,\ldots]$ (or Picard-Fuchs system) which annihilates the period $\psi_0(z)$. 
By \cref{01122025zutopia} we know that the Hasse-Witt invariant $a(z)$ satisfies also $L(a(z))=0$. A priori, $a(z)$ might have poles along $z_i=0$, and other degeneracy loci of $X_z$. We must prove that this is not the case, and hence $a(z)$ is a polynomial in $z$. 
Since $0$ is a MUM point, we know that there is a unique meromorphic solution $\psi_0(z)$ to $L=0$.
From all these we might try to get the desired statement. We  must further argue that this polynomial has degree $\leq p-1$.   

\section{Mirror quintic}
Recall that in \cite[Section 3]{GMCD-MQCY3}, the third author constructed the moduli $\sf S$ of pairs $(X, \alpha)$, where $X$ is a mirror quintic and $\alpha$ is a holomorphic differential 3-form on $X$. He further showed that
\[
\Ts_\Q \cong \mathrm{Spec}(\mathbb{Q}[t_1, t_5, \frac{1}{(t_1^5-t_5)t_5}])
\]
and the universal family over $\Ts_\Q$ is given by a desingularization of $X_t: \P\{f=0\}/G$, where
\[
f(x) = -t_5x_0^5 - x_1^5 - x_2^5 - x_3^5 - x_4^5+5t_1x_0x_1x_2x_3x_4
\]
and $\alpha$ in the affine coordinate $x_4=1$ is given by
\[
\alpha= \frac{dx_0 \wedge dx_1 \wedge dx_2} 
{\partial f/\partial x_3}.
\]
The holomorphic period of $\alpha$ along the 3-torus \[\delta_{0,z} = \{(x_0, x_1, x_2, x_3, x_4) \mid |x_0| = |x_1| = |x_2| = \delta, |x_3| \ll 1, x_4=1\}\] for $t_1=1, t_5=z$ 
is given by
$$
5 (2\pi i)^{-3} \int_{\delta_{0,z}} \alpha_z = \sum^{\infty}_{k=0} \frac{(5k)!}{(k!)^5} (\frac{z}{5^5})^k, \quad |z|<1.
$$

More generally, we may consider an arbitrary Dwork family of n-folds in $\mathbb{P}^{n+1}$, together with a choice of holomorphic $n$-form $\alpha$.\cite{GMCD-DF} This leads us to the family 
\[
X_{t_0,t_{n+2}}:f=0,\ \  
f=-t_{n+2}x_0^{n+2} - x_{1}^{n+2}  - \cdots - x_{n+1}^{n+2} + (n+2)t_1x_0x_1\cdots x_{n+1}.
\] 
\[ 
\alpha= \frac{dx_0 \wedge \cdots \wedge dx_{n}}{f_{n}}, f_{n} = \frac{\partial f}{\partial x_{n}},\  \ \hbox{ in the affine coordinate } x_{n+1}=1
\]
The holomorphic period of $\alpha$ for $t_1=1, t_{n+2}=z$  is given by
$$
F(z) = (n+2) (2\pi i)^{-n} \int_{\delta_{0,z}} \alpha_z = \sum^{\infty}_{k=0} \frac{((n+2)k)!}{(k!)^{n+2}} (\frac{z}{(n+2)^{n+2}})^k, \quad |z|<1.
$$


\begin{prop} We let $p$ prime  $p\not|n+2$. Then:
\begin{equation}
C(\alpha) = \left( \sum_{k=0}^{\lfloor \frac{p-1}{n+2} \rfloor} \frac{((n+2)k)!}{(k!)^{n+2}}
t_{n+2}^k ((n+2)t_1)^{p-1-(n+2)k} \right)^{\frac{1}{p}} \alpha
\end{equation}
\end{prop}

\begin{proof}
Choose the affine open part $x_0 = 1$ of the moduli for Calabi-Yau n-folds. Based on the computational result of Cartier operators of Miller in \cite[Corollary 1]{Miller72}, we get
\begin{equation}
C(\omega) = \frac{\psi(f^{p-1}x_1x_2\cdots x_{n+1})}{x_1x_2\cdots x_{n+1}} \omega,
\end{equation}
where $\psi$ the $p^{-1}$-linear operator defined over $\mathbb{Q}[t_1, t_{n+2}, \frac{1}{(n+2)(t_1^{n+2}-t_{n+2})}][x_1,x_2,\cdots,x_{n+1}]$ given by
\[
\psi(x_1^{i_1}x_2^{i_2}\cdots x_{n+1}^{i_{n+1}}) = \begin{cases}
x_1^{i_1/p}x_2^{i_2/p}\cdots x_{n+1}^{i_{n+1}/p}, & \mathrm{if} \ p|i_j, j=1,2,\cdots, i_{n+1}; \\
0, & \mathrm{otherwise}.
\end{cases}
\]
Note that the coefficient $(x_1x_2\cdots x_{n+1})^{p-1}$ in $f^{p-1}$ is
\begin{equation}
    \begin{split}
  &\sum_{k=0}^{\lfloor \frac{p-1}{n+2} \rfloor}\frac{(p-1)!}{(k!)^{n+2}(p-1-(n+2)k)!} (-1)^{(n+2)k}t_{n+2}^k ((n+2)t_1)^{p-1-(n+2)k} \\
  =& \sum_{k=0}^{\lfloor \frac{p-1}{n+2} \rfloor}\frac{((n+2)k)!}{(k!)^{n+2}}
\binom{p-1}{(n+2)k} (-1)^{(n+2)k} t_{n+2}^k ((n+2)t_1)^{p-1-(n+2)k}.      
    \end{split}
\end{equation}

Together with the identity 
$$
\binom{p-1}{(n+2)k}\equiv_p (-1)^{(n+2)k},
$$
we get the final result.
\end{proof}

The logarithmic period is given by 
$$
(n+2)(2\pi i )^{-n}\int_{\delta_{1,z}}\alpha_z=F(z)
\ln(\frac{z}{(n+2)^{n+2}})+G(z),\ \ 
$$
$$
G(z):= (n+2)\sum_{k=1}^{\infty}  \frac{((n+2)k)!}{(k!)^{n+2}} (\sum_{j=k+1}^{(n+2)k}\frac{1}{j}) (\frac{z}{(n+2)^{n+2}})^k.
$$
The period expression of $t_1$ and $t_{n+1}$ up to some $2\pi i$ factors
are 
$$
t_1=F(z),\ t_{n+2}=zF(z)^{n+2}. 
$$
After inserting mirror map $z(q)$, we get the $q$-expansion of these quantities. For example, when $n=1$, we get the reversion formula:
\[
t_1(q) = \ _2F_1(\frac{1}{3},\frac{2}{3},1;z(q)) = \theta_3(q)\theta_3(q^3)+\theta_2(q)\theta_2(q^3)
\]
and
\[ t_3(q) = z(q) _2F_1(\frac{1}{3},\frac{2}{3},1;z(q)) = 27\frac{\eta^9(q^3)}{\eta^3(q)},\]
whose $q$-expansions are:
$$
t_1(q)= 1+6q+6q^3+6q^4+O\left(q^6\right),\ \   t_3(q)= 27q+81q^2+243q^3+351q^4+729q^5+O\left(q^6\right).\ \ 
$$
See \cite[Section 8]{GMCD-DF} and \cite[Section 5.1]{Younes2019}. These formula has been shown in \cite[Corollary 2.4]{BBG1994} via the cubic theta functions.

For $n=2$, the mirror map is given by
\[
z(q) = q-104q^2+6444q^3-311744q^4+13018830q^5 + O(q^6),
\]
which can be found in \cite[(7.3.7)]{Gannon2023} and \cite[A286329]{oeis}.  Then we have:
$$
t_1(q)=\ _3F_2(\frac{1}{4},\frac{2}{4}, \frac{3}{4};1,1| z(q)) =  \theta_{D_4}^4(q) = 1+24q+24q^2+96q^3+24q^4+144q^5+O(q^6),   
$$
where $\theta_{D_4}(q)$ is the theta series of $D_4$-lattice and 
$$
t_4(q) = z(q) _3F_2(\frac{1}{4},\frac{2}{4}, \frac{3}{4};1,1| z(q))^4 = c \eta^8(q) \eta^8(q^2) = c(q-8q^2+12q^3+64q^4-210q^5+O(q^6)),
$$
where $\eta(q)$ is the eta function and $c$ is a constant. 
For the q-expansion of $t_1(q)$ (resp. $t_4(q)$), see \cite[A004011]{oeis} (resp. \cite[A002288]{oeis}). These results are coincide with the q-expansion computations in \cite{GMCD-DF}. 
\

For $n=3$ we get
\begin{align}
    \begin{split}
        t_1(q)&=1+120 q+21000 q^2+14115000 q^3+13414125000 q^4+15234972675120
   q^5+O\left(q^6\right)\\
   t_5(q) &= q-170 q^2-41475 q^3-32183000 q^4-32678171250 q^5+O\left(q^6\right)\\
    \end{split}
\end{align}
which appear in \cite{GMCD-MQCY3}. 

For $n=4$ we get 
$$
t_1=\frac{1}{6}+20q+82620q^2+O(q^3),\   t_6=\frac{1}{46656}q-\frac{1}{24}q^2 + O(q^2),\ \ 
$$
see \cite[Section 8.3]{GMCD-DF}.

For the purpose of studying supersingular Calabi-Yau threefolds, we have also listed some Hasse-Witt invariants of the mirror quintic in Table~\ref{table: A_p(x,y) of mirror quintic at small primes} as product of irreducible polynomials over $\mathbb{F}_p[x,y]$, where $x:=t_1,\ y:=t_{n+2}$. 
\begin{table}
\begin{equation*}
\begin{array}{|c|l|}
\hline
 p & A_p(x,y) \\
 \hline
 2 & x \\[2pt]
 3 & x^2 \\[2pt]
 5 & x^4 \\[2pt]
 7 & x \left(x^5+y\right) \\[2pt]
 11 & x^{10}+10 x^5 y+y^2 \\[2pt]
 13 & x^2 \left(x^{10}+3 x^5 y+y^2\right) \\[2pt]
 17 & x \left(x^5+16 y\right) \left(x^{10}+2 x^5 y+12 y^2\right) \\[2pt]
 19 & x^3 \left(x^{15}+6 x^{10} y+8 x^5 y^2+7 y^3\right) \\[2pt]
 23 & x^2 \left(x^5+14 y\right) \left(x^5+21 y\right) \left(x^{10}+16 x^5 y+7 y^2\right) \\[2pt]
 29 & x^3 \left(x^{10}+15 y^2\right) \left(x^{15}+4 x^{10} y+24 x^5 y^2+14 y^3\right) \\[2pt]
 31 & \left(x^5+4 y\right) \left(x^5+17 y\right) \left(x^{20}+6 x^{15} y+25 x^{10} y^2+3 x^5 y^3+y^4\right) \\[2pt]
 37 & x \left(x^5+9 y\right) \left(x^5+23 y\right) \left(x^{25}+14 x^{20} y+6 x^{15} y^2+21 x^{10} y^3+17 x^5 y^4+3 y^5\right) \\[2pt]
 41 & \left(x^5+5 y\right) \left(x^{10}+26 x^5 y+27 y^2\right) \left(x^{25}+7 x^{20} y+30 x^{15} y^2+31 x^{10} y^3+12 x^5 y^4+31 y^5\right) \\[2pt]
 43 & x^2 \left(x^{40}+34 x^{35} y+9 x^{30} y^2+31 x^{25} y^3+20 x^{20} y^4+2 x^{15} y^5+41 x^{10} y^6+11 x^5 y^7+29 y^8\right) \\[2pt]
 47 & x \left(x^5+21 y\right) \left(x^{10}+23 x^5 y+17 y^2\right) \left(x^{10}+44 x^5 y+23 y^2\right) \left(x^{20}+32 x^{15} y+25 x^{10} y^2+15 x^5 y^3+3 y^4\right) \\[2pt]
 \hline
\end{array}
\end{equation*}
\caption{ $A_p(x,y)$ of mirror quintic at small primes $p$.}
\label{table: A_p(x,y) of mirror quintic at small primes}
\end{table}

In a similar way we compute the Hasse-Witt invariant of the rest three families of CY threefolds in \cref{28112025bri}. 

\begin{prop} 
\begin{enumerate}
    \item 
For the hypersurface in the weighted projective space $\mathbb{P}^{5,2,1,1,1}$ defined by
\[
-x_0^2 - x_1^5 - t_{10}x_2^{10} - x_3^{10} - x_4^{10} +t_1 x_0x_1x_2x_3x_4 = 0,
\]
we have:
\begin{equation}
C(\alpha) = \left( \sum_{k=0}^{\lfloor \frac{p-1}{10} \rfloor} \frac{(10k)!}{(k!)^3(2k)!(5k)!}t_{10}^k t_1^{p-1-10k} \right)^{\frac{1}{p}} \alpha,
\end{equation}
where $p$ is a prime number such that $p\not|10$.
\item
For the hypersurface in the weighted projective space $\mathbb{P}^{4,1,1,1,1}$ defined by
\[
-x_0^2 -t_8 x_1^8 - x_2^{8} - x_3^{8} - x_4^{8} +t_1 x_0x_1x_2x_3x_4 = 0,
\]
we have:
\begin{equation}
C(\alpha) = \left( \sum_{k=0}^{\lfloor \frac{p-1}{8} \rfloor} \frac{(8k)!}{(k!)^4(4k)!}t_{8}^k t_1^{p-1-10k} \right)^{\frac{1}{p}} \alpha,
\end{equation}
where $p$ is a prime number such that $p\not| 2$.
\item 
For the hypersurface in the weighted projective space $\mathbb{P}^{1,1,1,2,1}  $ defined by
\[
-t_6x_0^6 - x_1^6 - x_2^{6} - x_3^{3} - x_4^{6} +t_1 x_0x_1x_2x_3x_4 = 0,
\]
we have:
\begin{equation}
C(\alpha) = \left( \sum_{k=0}^{\lfloor \frac{p-1}{6} \rfloor} \frac{(6k)!}{(k!)^4(2k)!}t_{6}^k t_1^{p-1-6k} \right)^{\frac{1}{p}} \alpha,
\end{equation}
where $p$ is a prime number such that $p\not|6$.
\end{enumerate}
\end{prop}

\begin{proof}
The proof is similar to the proof in Proposition 1. Let's take the first one as an example. Choose the affine open part $x_0 = 1$ of the first family for Calabi-Yau 3-folds. Based on the formula
\begin{equation}
C(\omega) = \frac{\psi(f^{p-1}x_1x_2x_3x_{4})}{x_1x_2x_3x_{4}} \omega,
\end{equation}
we only need to compute the coefficient $(x_1x_2x_3x_4)^{p-1}$ in $f^{p-1}$, which is
\[
\sum_{k=0}^{\lfloor \frac{p-1}{10} \rfloor}\frac{(p-1)!}{(k!)^{3}(2k)!(5k)!(p-1-(10k)!)} (-1)^{10k}t_{10}^k t_1^{p-1-10k}= \sum_{k=0}^{\lfloor \frac{p-1}{10} \rfloor}\frac{(10k)!}{(k!)^{3}(2k)!(5k)!} t_{10}^k t_1^{p-1-10k}.
\]
Hence we have:
\[
C(\alpha) = \left( \sum_{k=0}^{\lfloor \frac{p-1}{10} \rfloor} \frac{(10k)!}{(k!)^3(2k)!(5k)!}t_{10}^k t_1^{p-1-10k} \right)^{\frac{1}{p}} \alpha
\]
in this case.
\end{proof}
\begin{rem}
We also have the following identities:
\begin{eqnarray*}
\frac{(\frac{1}{5})_k (\frac{4}{5})_k (\frac{2}{5})_k (\frac{3}{5})_k }{k!^4} &=& \frac{1}{5^{5k}} \frac{(5k)!}{(k!)^5} \\
\frac{(\frac{1}{10})_k (\frac{9}{10})_k (\frac{3}{10})_k (\frac{7}{10})_k }{k!^4} &=& \frac{1}{10^{3k}5^{2k}2^{5k}} \frac{(10k)!}{(k!)^3 (2k)! (5k)!} \\
\frac{(\frac{1}{8})_k (\frac{7}{8})_k (\frac{3}{8})_k (\frac{5}{8})_k }{k!^4} &=& \frac{1}{8^{4k}2^{4k}} \frac{(8k)!}{(k!)^4 (4k)!} \\
\frac{(\frac{1}{6})_k (\frac{5}{6})_k (\frac{2}{6})_k (\frac{4}{6})_k }{k!^4} &=&\frac{1}{6^{4k}3^{2k}} \frac{(6k)!}{(k!)^4 (2k)!}.
\end{eqnarray*}
\end{rem}

\section{Calabi-Yau Operators}

The function $A_p(z)$, as defined in \cref{17112025ymsc}, can be computed directly from the periods. It follows that specific cases of the conjecture can be tested without any reference to an underlying family of varieties. We test this conjecture for a list of 545 Calabi-Yau operators and find that it fails sometimes. This list of Calabi-Yau opearators consists mainly of operators from the AESZ list \cite{alenstzu} and can be found as an attached file, along with a SageMath Jupyter notebook.

A Calabi-Yau operator is a fourth order linear differential operator in a variable $z$, which satisfies a number of properties that make it a candidate Picard-Fuchs equation for a family of Calabi-Yau threefolds \cite{alenstzu,almzud,vanStratenCYOperators}. These properties include the requirement that $z=0$ is a point of maximal unipotent monodromy (MUM), the mirror map $z(q)$ defined by this MUM point has integer coefficients when expanded in $q$, the genus $0$ instanton numbers are integers up to some overall multiple, etc.

There is no guarantee that all of the Calabi-Yau operators in the AESZ list \cite{alenstzu} are really the Picard-Fuchs operator for a family of Calabi-Yau threefolds with $h^{2,1}=1$, so these examples do not immediately invalidate Conjecture~\ref{17112025ymsc}. It is tempting to speculate that, for a given Calabi-Yau operator, the failure of Conjecture~\ref{17112025ymsc} rules out the existence of an associated smooth Calabi-Yau threefold with $h^{2,1}=1$.


For each operator, we compute the holomorphic period
\begin{equation*}
    \varpi_0(z) = \sum_{n=0}^{\infty}c_nz^n
\end{equation*}
around $z=0$ and normalise it so that $c_0=1$.


From a list of 545 Calabi-Yau operators, we find that 460 operators satisfy
\begin{equation*}
    A_p(z) =\varpi_0(z)^{p-1}~\sum_{n=0}^{p-1}c_nz^n    \equiv_p1 +O(z^{600})
\end{equation*}
for the first $100$ primes. The truncation at $O(z^{600})$ is an arbitrarily chosen cutoff.

The remaining 85 operators do not satisfy this condition. 
However, we still find that
\begin{equation*}
    A_p(z) \equiv_p1 +O(z^{p})
\end{equation*}
for all of the operators that we check and the first $100$ primes.

When Conjecture~\ref{17112025ymsc} fails, it often fails in interesting ways. For example, consider the operator 
\begin{align*}
\label{eq: AESZ 1hat after change of variables}
    \begin{split}
        \mathcal{L}=&\theta^4-2\cdot 5(10000\theta^4 + 12500\theta^3 + 9500\theta^2 + 3250\theta + 399)\\
        &+2^2\cdot5^8z^2(2400\theta^4 + 6000\theta^3 + 6290\theta^2 + 2800\theta + 399)\\
        &-2^5\cdot 5^{14}z^3(4\theta + 3)(80\theta^3 + 240\theta^2 + 221\theta + 42)\\
        &+2^4\cdot5^{20}z^4(4\theta + 1)(4\theta + 3)(4\theta + 7)(4\theta + 9)
    \end{split}
\end{align*}
where $\theta=z\frac{d}{dz}$. This operator can be obtained from operator $\widehat{1}$ in \cite{alenstzu} by making the change of variables $z\mapsto 2z$. The change of variables is needed to ensure that the holomorphic period has integer coefficients.

The singularities of this operator are summarised by the Riemann $\mathcal{P}$-symbol
\begin{equation*}
    \mathcal{P}\left\{\begin{array}{c c c} 0 & \frac{1}{2^3\cdot 5^5} & \infty\\[1pt]
    \hline
    0 & \frac{1}{20} & \frac{1}{4}\\
    0 & \frac{3}{20} & \frac{3}{4}\\
    0 & \frac{7}{20} & \frac{7}{4}\\
    0 & \frac{9}{20} & \frac{9}{4}\\
    \end{array}\right\},
\end{equation*}
which is somewhat exotic because it does not have a conifold point (a singularity with indices $(0,1,1,2)$).

The holomorphic period for this operator is given by
\begin{align*}
    \begin{split}
        \varpi_0(z)=&1+3990 z+49934850 z^2+806586406500 z^3+14635392749853750 z^4+O\left(z^{5}\right)
    \end{split}
\end{align*}
which we use to compute $A_p(z)$ for many primes and to a high order in $z$.

After some experimentation, we are led to the following conjecture, which we verify for the first $120$ primes and to $O(z^{10000})$.

\begin{conj}
\label{01022026aquamatica}
If $p$ is an inert rational prime in $\mathbb{Q}\left(\sqrt{-5}\right)$, then
\begin{equation*}
    A_p(z)\equiv_p \left(\sqrt{1-2^3\cdot5^5 z}\right)^p
\end{equation*}
where the right hand side should be understood as a power series in $z$.
\end{conj}

Conjecture~\ref{01022026aquamatica} is a result of computer experiments. In order to test Conjecture~\ref{01122025zutopia}, we  computed $A_p(z)$ for many primes and to a high order in $z$. A few examples at small primes are listed in Table~\ref{table: A_p(z) at small primes for pullback of operator 1hat}.  It quickly becomes clear that, for half of the primes, $A_p(z)$ contains only $p^{th}$ powers of $z$. In other wrds,
\begin{equation*}
    A_p(z)\equiv_p f_p(z)^p.
\end{equation*}
for some $f_p(z)=\mathbb{F}_p\llbracket z\rrbracket$.
With enough primes, we are able to conjecture that $A_p(z)$ is a $p^{th}$ power whenever $p$ is inert in $\mathbb{Q}\left(\sqrt{-5}\right)$.

At this point, one naturally wonders if the $f_p(z)$ are the mod $p$ reduction of some $f(z)\in\mathbb{Z}\llbracket z\rrbracket$. It is straightforward to search for a candidate $f$ by applying the Chinese remainder theorem term-wise to $f_p$ for many primes. This defines an element of $\left(\mathbb{Z}/N\mathbb{Z}\right)\llbracket z\rrbracket$ for some large $N\in\mathbb{N}$. We then apply the rational reconstruction algorithm term-wise to this power series to find a power series with integer coefficients, which stabilises as we consider more primes and increase $N$. We recognize the resulting power series as the power series expansion of $\sqrt{1-2^3\cdot5^5 z}$. The computer data used in this section can be found in \href{https://w3.impa.br/~hossein/WikiHossein/files/Singular%20Codes/25_02_2026_HW_For_CY/}{third author's webpage.}

Observations similar to Conjecture~\ref{01022026aquamatica} have appeared in \cite{Ducker:2025wfl}, where the authors study a number of operators like the operator $\mathcal{L}$ in equation~\eqref{01022026aquamatica}. These operators are obtained as pullbacks of fifth-order differential operators \cite{almzud,2006math.....12215A}. 

The authors of \cite{Ducker:2025wfl} note that, in specific examples, the fourth order operator cannot have a basis of solutions with rational monodromy and that this is remedied by passing to a double cover of the $z$-plane with quadratic branch points introduced at the two non-zero singularities. In our example, this would make the function $\sqrt{1-2^3\cdot5^5 z}$ single valued. They also note that, without passing to the double cover, the deformation method for computing local zeta functions from Picard-Fuchs equations fails \cite{Candelas:2021tqt}. These observations support the idea that the failure of Conjecture~\ref{17112025ymsc} rules out the existence of an associated smooth Calabi-Yau threefold with $h^{2,1}=1$.


Finally, we note that a single Calabi-Yau operator can have multiple MUM points. The standard example is the Picard-Fuchs operator described by Rodland in \cite{1998math......1092R}. This operator appears as operator 27 in \cite{almzud} and is given by
\begin{align*}
    \mathcal{L}_{Rodland}=&3^{2} \theta^4-3 z\left(173\theta^4+340\theta^3+272\theta^2+102\theta+15\right)\\
    &-2 z^{2}\left(1129\theta^4+5032\theta^3+7597\theta^2+4773\theta+1083\right)\\
    &+2 z^{3}\left(843\theta^4+2628\theta^3+2353\theta^2+675\theta+6\right)\\
    &-z  ^{4}\left(295\theta^4+608\theta^3+478\theta^2+174\theta+26\right)\\
    &+z^{5}\left(\theta+1\right)^4
\end{align*}
It has the Riemann symbol
\begin{equation*}
    \mathcal{P}\left\{\begin{array}{c c c c} -3 & 0 & \omega_i & \infty\\[1pt]
    \hline
    0 & 0 & 0 & 1\\
    1 & 0 & 1 & 1\\
    3 & 0 & 1 & 1\\
    4 & 0 & 2 & 1\\
    \end{array}\right\},
\end{equation*}
\noindent where $\omega_i$ is any one of the solutions of $\omega_i^3-289\omega_i^2-57\omega_i+1=0$. There are MUM points at $z=0$ and $z=\infty$ with holomorphic solutions 
\begin{align*}
    \begin{split}
        \varpi_0^{(0)}(z)&=1+5 z+109 z^2+3317 z^3+121501 z^4+4954505
   z^5+216867925 z^6+O\left(z^{7}\right)\\
   \varpi_0^{(\infty)}(z)&=1+ \frac{17}{z}+ \frac{1549}{z^2}+ \frac{215585}{z^3}+ \frac{36505501}{z^4}+
   \frac{6921832517}{z^5}+ \frac{1412721479989}{z^6}+O\left(z^{7}\right)
    \end{split}
\end{align*}
We find that the Hasse-Witt invariants computed at each of these MUM points are related by $z\rightarrow\frac{1}{z}$. For example, at $p=7$, we find
\begin{align*}
    A_7^{(0)}(z)&=z^6+4 z^5+2 z^4+2 z^2+4 z+1\\
    A_7^{(\infty)}(z)&=\frac{1}{z^6}+\frac{4}{z^5}+\frac{2}{z^4}+\frac{2}{z^2}+\frac{4} {z}+1~.
\end{align*}

\begin{table}
\begin{equation*}
\begin{array}{|c|l|}
\hline
 p & A_p(z) \\[2pt]
 \hline
2 & 1 + O(z^{100}) \\[4pt]
3 & 1 + O(z^{100}) \\[4pt]
5 & 1 + O(z^{100}) \\[4pt]
7 & 1 + O(z^{100}) \\[4pt]
11 & 1 + 7 z^{11} + 3 z^{22} +  z^{33} + 5 z^{44} + 6 z^{55} + 3 z^{66} + O(z^{100}) \\[4pt]
13 & 1 + 6 z^{13} + 8 z^{26} + 4 z^{39} + 9 z^{52} + 5 z^{65} + 7 z^{78} + 12 z^{91} + O(z^{100}) \\[4pt]
17 & 1 + 12 z^{17} + 13 z^{34} + 14 z^{51} + 11 z^{68} + 9 z^{85} + O(z^{100}) \\[4pt]
19 & 1 + 2 z^{19} + 17 z^{38} + 4 z^{57} + 9 z^{76} + 9 z^{95} + O(z^{100}) \\[4pt]
23 & 1 + 22 z^{23} + 12 z^{24} + 15 z^{25} + 7 z^{26} +  z^{27} + 8 z^{29} + 20 z^{30} + 15 z^{31} + 12 z^{32} + 5 z^{33} + 22 z^{34} + 19 z^{35}\\
& ~+ 15 z^{36} + 9 z^{37} + 21 z^{38} + 15 z^{39} +  z^{40} + 22 z^{41} + 3 z^{42} + 3 z^{43} +  z^{44} + 8 z^{45} + 4 z^{46} + 12 z^{47} + 19 z^{48}\\
& ~+  z^{49} + 6 z^{50} + 11 z^{51} + 18 z^{52} +  z^{53} + 11 z^{54} + 18 z^{55} + 5 z^{56} +  z^{57} + 4 z^{58} + 8 z^{59} + 9 z^{60} + 2 z^{61} + 19 z^{62}\\
&~+ 15 z^{63} + 19 z^{64} + 2 z^{65} + 16 z^{66} + 4 z^{67} + 18 z^{68} + 7 z^{69} + 4 z^{70} +  z^{71} + 21 z^{72} + 11 z^{73} + 6 z^{74} + 13 z^{75}\\
& ~+ 18 z^{76} + 16 z^{77} + 21 z^{78} + 21 z^{79} + 8 z^{80} + 3 z^{81} + 6 z^{82} + 10 z^{83} + 15 z^{84} + 22 z^{85} + 4 z^{86} + 17 z^{87} +  z^{88}\\
& ~+ 10 z^{89} + 22 z^{90} + 18 z^{91} + 9 z^{92} + 19 z^{93} + 18 z^{94} + 18 z^{96} + 21 z^{97} +  z^{98} + 8 z^{99} + O(z^{100}) \\[4pt]
29 & 1 + 13 z^{29} + 8 z^{30} + 5 z^{31} + 21 z^{32} + 24 z^{33} + 15 z^{34} + 17 z^{35} + 2 z^{36} + 28 z^{37} + 8 z^{38} + 24 z^{39} + 12 z^{40}\\
&~+ 18 z^{41} + 20 z^{43} + 16 z^{44} + 15 z^{45} + 26 z^{46} + 22 z^{47} + 4 z^{48} + 22 z^{49} + 15 z^{50} + 15 z^{51} + 26 z^{52} + 28 z^{53}\\
&~+ 3 z^{54} + 8 z^{55} + 26 z^{56} + 14 z^{57} + 9 z^{58} + 5 z^{59} +  z^{60} + 13 z^{61} + 14 z^{62} + 9 z^{63} + 2 z^{64} + 10 z^{65} + 21 z^{66}\\
&~+ 24 z^{67} + 10 z^{68} + 10 z^{70} + 26 z^{71} + 17 z^{72} + 14 z^{73} + 2 z^{74} +  z^{75} + 16 z^{76} + 10 z^{77} + 8 z^{78} + 22 z^{79} + 7 z^{80}\\
&~+ 11 z^{81} + 9 z^{82} + 8 z^{83} + 2 z^{84} + 8 z^{85} + 21 z^{86} + 4 z^{87} + 27 z^{88} + 18 z^{89} + 4 z^{90} + 14 z^{91} + 21 z^{92} + 19 z^{93}\\
&~ + 28 z^{94} + 12 z^{95} + 25 z^{96} + 6 z^{97} + 14 z^{98} + 9 z^{99} + O(z^{100}) \\[4pt]
31 & 1 + 24 z^{31} + 22 z^{62} + 30 z^{93} + O(z^{100}) \\[4pt]
37 & 1 + 6 z^{37} + 19 z^{74} + O(z^{100}) \\[4pt]
41 & 1 + 34 z^{41} + 40 z^{42} + 19 z^{43} + 15 z^{45} + 28 z^{46} + 28 z^{47} + 33 z^{48} + 36 z^{49} + 5 z^{50} + 40 z^{51} + 19 z^{52} + 28 z^{53}\\
& ~+ 38 z^{54} + 33 z^{55} + 13 z^{56} + 22 z^{57} + 25 z^{58} + 21 z^{59} + 38 z^{60} + 11 z^{61} + 39 z^{62} + 31 z^{63} +  z^{64} + 19 z^{65}\\
&~ + 31 z^{66} + 10 z^{67} + 8 z^{68} + 4 z^{69} + 6 z^{70} + 35 z^{71} + 32 z^{72} + 9 z^{73} + 39 z^{74} + 38 z^{75} +  z^{76} + 12 z^{77} + 4 z^{78}\\
&~+ 37 z^{79} + 30 z^{80} + 9 z^{81} + 4 z^{82} + 38 z^{83} + 29 z^{84} + 16 z^{85} + 29 z^{86} + 35 z^{87} + 5 z^{88} + 12 z^{89} + 25 z^{90} + 35 z^{91}\\
&~ + 24 z^{92} + 27 z^{93} + 6 z^{94} + 27 z^{95} + 33 z^{96} + 8 z^{97} + 10 z^{98} + 25 z^{99} + O(z^{100}) \\[4pt]
43 & 1 + 40 z^{43} + 35 z^{44} + 11 z^{45} + 22 z^{46} + 34 z^{47} + 20 z^{48} + 6 z^{49} + 40 z^{50} + 29 z^{51} + 14 z^{52} +  z^{53} + 28 z^{54}\\
&~+ 30 z^{55} + 14 z^{56} + 16 z^{57} + 5 z^{58} +  z^{59} + 12 z^{60} + 40 z^{61} + 7 z^{62} + 7 z^{63} + 25 z^{64} + 18 z^{65} + 34 z^{66} + 9 z^{67}\\
&~+ 27 z^{68} + 28 z^{69} + 19 z^{70} + 6 z^{71} + 14 z^{72} + 29 z^{73} + 29 z^{74} + 21 z^{75} + 23 z^{76} + 8 z^{77} + 12 z^{78} + 39 z^{79} + 35 z^{80}\\
&~+ 33 z^{81} + 26 z^{82} + 7 z^{83} + 27 z^{84} + 30 z^{85} + 38 z^{86} + 42 z^{87} + 10 z^{88} + 6 z^{89} + 42 z^{90} + 34 z^{91}+ 36 z^{92} + 4 z^{93}\\
&~+ 20 z^{94} + 4 z^{95} + 25 z^{96} + 26 z^{97} + 10 z^{98} + 20 z^{99} + O(z^{100}) \\[4pt]
47 & 1 + 17 z^{47} + 22 z^{48} + 37 z^{49} + 46 z^{50} + 5 z^{51} + 10 z^{52} + 32 z^{53} + 2 z^{54} + 38 z^{55} + 25 z^{56} + 32 z^{57} + 28 z^{58}\\
&~+ 27 z^{59} + 22 z^{60} + 25 z^{61} + 40 z^{62} + 37 z^{63} + 31 z^{64} + 15 z^{65} + 13 z^{66} + 45 z^{67} + 11 z^{68} + 26 z^{69} + 18 z^{70}\\
&~+ 20 z^{71} + 10 z^{72} + 5 z^{73} + 3 z^{74} + 26 z^{75} + 3 z^{76} + 5 z^{77} + 40 z^{78} + 37 z^{79} + 45 z^{80} + 4 z^{81} + 22 z^{82} +  z^{83}\\
&~+ 39 z^{84} + 7 z^{85} + 3 z^{86} + 4 z^{87} + 3 z^{88} + 21 z^{89} + 33 z^{90} + 27 z^{91} + 20 z^{92} + 30 z^{94} + 15 z^{95} + 10 z^{96} + 24 z^{97}\\
&~+ 13 z^{98} + 43 z^{99} + O(z^{100}) \\[4pt]
53 & 1 + 8 z^{53} + O(z^{100}) \\[4pt]
59 & 1 + 8 z^{59} + O(z^{100}) \\[4pt]
\hline
\end{array}
\end{equation*}
\caption{ $A_p(z)$ of $\mathcal{L}$ for small primes $p$.}
\label{table: A_p(z) at small primes for pullback of operator 1hat}
\end{table}

\newpage
\bibliography{biblio.bib}
\bibliographystyle{alpha}

\printindex

\end{document}